\documentclass[a4paper]{article}

\usepackage{graphicx}
\usepackage{amsmath,amsfonts,amssymb,amsthm,amscd,enumerate}
\usepackage[all]{xy}
\usepackage[active]{srcltx}

\usepackage{cite}
\usepackage{color}

\newtheorem{thm}{Theorem} 

\theoremstyle{definition}
\newtheorem{defn}{Definition}

\newcommand{\nc}[2]{\newcommand{#1}{#2}}
\newcommand{\rnc}[2]{\renewcommand{#1}{#2}}
\rnc{\[}{\begin{equation}}
\rnc{\]}{\end{equation}}
\nc{\wegengruen}{\end{equation}}

\newcommand{\Z}{\mathbb{Z}}
\newcommand{\N}{\mathbb{N}}

\newcommand{\C}{\mathbb{C}}
\newcommand{\hsp}{{\hspace{-1pt}}}
\newcommand{\hs}{{\hspace{1pt}}}

\newcommand{\hH}{\mathcal{H}}    
\newcommand{\K}{{\mathcal{K}}}
\newcommand{\G}{{\mathcal{G}}}
\newcommand{\cO}{\mathcal{O}}

\newcommand{\cN}{\mathcal{N}}

\newcommand{\CCq}{\cO(\C_q^2)}  
\newcommand{\dd}{\mathrm{d}}

\newcommand{\im}{\mathrm{i}}
\newcommand{\E}{\mathrm{e}}

\newcommand{\id}{{\mathrm{id}}}

\newcommand{\Ker}{\mathrm{ker}}
\newcommand{\ran}{\mathrm{ran}}

\newcommand{\spec}{\mathrm{spec}}
\newcommand{\dom}{\mathrm{dom}}
\newcommand{\supp}{\mathrm{supp}}
\newcommand{\rmB}{\mathrm{B}}

\newcommand{\alg}{\text{*-}\mathrm{alg}}
\newcommand{\LL}{\mathcal{L}_2}
\newcommand{\lN}{\ell_2(\N)}

\newcommand{\ra}{\rightarrow}

\newcommand{\ot}{\otimes}

\title{Function algebras on a 2-dimensional quantum complex plane}

\author{{\sc Ismael Cohen} \\
\normalsize
Instituto de F\'isica y Matem\'aticas\\
\normalsize
Universidad Michoacana de San Nicol\'as de Hidalgo, Morelia, M\'exico\\[2pt]
\normalsize
and\\[2pt]
\normalsize
Centro de Ciencias Matem\'aticas, Campus Morelia\\
\normalsize
Universidad Nacional Aut\'onoma de M\'exico (UNAM), Morelia, M\'exico\\
\normalsize
e-mail: {\it ismaelcohen10@gmail.com}\\[16pt] 
{\sc Elmar Wagner\footnote{
corresponding author \ \  {\it MSC2010:} 46L85, 46L52   \ \  
}
} \\
\normalsize
Instituto de F\'isica y Matem\'aticas\\
\normalsize
Universidad Michoacana de San Nicol\'as de Hidalgo, Morelia, M\'exico\\
\normalsize
e-mail: {\it elmar@ifm.umich.mx}}

\usepackage{geometry}
\date{}                                       

\begin{document}
\maketitle

\begin{abstract}
The well-behaved representations of the coordinate algebra 
of a 2-dimensional quantum complex plane are classified 
and a C*-algebra is defined which can be viewed as the algebra of continuous functions 
on the 2-dimensional quantum complex plane
vanishing at infinity. 
\end{abstract}

\section{Introduction}
The general purpose of this paper is to study non-compact quantum spaces in the C*-algebraic framework. 
Usually quantum spaces arising in Quantum Group Theory are given by generators and relations.  
The (*-)algebra obtained in this way can then be viewed as the coordinate ring of polynomial functions on the quantum space. 
For compact quantum spaces, there is a general procedure to assign a unital C*-algebra to the quantum space: 
one considers the universal C*-norm defined as the supremum of the operator norms of all bounded *-representation 
of the coordinate ring and 
takes the closure with respect to this norm. Here, having a compact quantum space is essentially 
synonymous to the existence of the universal C*-norm. 

The non-compact situation is characterized by the fact that the *-algebra admits unbounded *-representations and that 
the universal C*-norm might not exist. For this setting, S.~L.~Woro\-nowicz developed a theory 
of C*-algebras generated by unbounded elements \cite{Wo91,Wo}. However, this method is not constructive, the unbounded operators 
and the C*-algebra have to be given at the beginning, one only proves that the unbounded operators actually generate 
the C*-algebra.  
Since we are more interested in having an explicit C*-algebra at hand than proving technical details, we prefer to construct 
a non-commutative C*-algebra 
by analogy to the classical C*-al\-ge\-bra of continuous functions vanishing at infinity 
on the corresponding locally compact 
space. The analogy to the classical case  involves concrete Hilbert space representations 
of the coordinate ring 
and can therefore be done only  by a case to case study. 
In the present paper, we will do it for a 2-dimensional quantum complex plane, 
the 1-dimensional version has already been treated in \cite{CW,So}. 
Similar construction of function algebras on non-compact quantum spaces can be found, for instance, 
in \cite{BS, KW, OW, SSV, SSV00} 
but none of these papers touches the  C*-algebra framework. 

Let us briefly outline our construction. First we classify all well-behaved Hilbert space representations 
of the coordinate ring $\CCq$. It is important to have knowledge about all possible 
representations because it turns out that different representations correspond to different domains 
of the quantum complex plane.  The next step is to realize these representations on a function space 
($\LL$-space) such that modulus of each generator 
(the non-negative self adjoint part in its polar decomposition)  
acts as a multiplication operator. Furthermore, the measures are chosen in such a way that 
the partial isometries from the polar decompositions are given on the same footing: they act 
as multiplicative $q$-shifts on functions. In this manner we obtain very simple commutation relations between 
the multiplication operators and the partial isometries. 
Then we consider an auxiliary *-algebra of bounded operators generated by continuous functions of  
the moduli of the generators (represented by multiplication operators) 
and powers of the partial isometries and their adjoints. 
For the interpretation as continuous functions on the 2-dimensional quantum complex plane 
vanishing at infinity, we require that the continuous functions belong to $C_0([0,\infty)\hsp\times\hsp [0,\infty))$ 
and that these functions, when evaluated at 0, do not depend on the phases 
(the partial isometries from the polar decompositions).  
Moreover, in order not to ``miss any points'', we consider some sort of universal representation, 
where the involved measures have the largest possible support.  
Finally, the C*-algebra 
of continuous functions vanishing at infinity is defined by 
taking the C*-closure of the auxiliary algebra in the operator norm. 

An advantage of our approach is that it allows a geometric interpretation of the different representations. 
As usual, a nontrivial 1-dimensional representation corresponds to a classical point, 
in our case to the origin of $\C^2_q$. Setting one generator to zero, we get a copy of $\C_q$ inserted into  
the quantum space $\C^2_q$.  Last but not least, there is a family of faithful representations 
that describe a 2-dimensional quantum complex plane, where the copy $\C_q$ from the previous representation 
is shrunk to a point. Therefore, restricting oneself (as  quite customary) 
to the family of faithful representations 
will not yield the whole 2-dimensional quantum complex plane.

\section{Preliminaries} 

Throughout this paper, $q$ stands for a real number in the interval $(0,1)$. 
The coordinate ring $\CCq$ of polynomial functions on 2-dimensional quantum complex plane  
is  the *-algebra over $\C$ 
generated by $z_1$ and $z_2$ satisfying the (overcomplete) relations \cite{KS} 
\begin{align}
z_2 z_1 &= q z_1 z_2\,, &
z_1^* z_2^* &= q z_2^* z_1^* \,,\label{R1}\\
 z_2 z_1^* &= q z_1^* z_2 \,,  &
z_1 z_2^* &= q z_2^* z_1\,,       \label{R2}\\
z_2 z_2^* &= q^2 z_2^* z_2\,, &
z_1 z_1^* &= q^2 z_1^* z_1 - (1-q^2) z_2^* z_2\,. \label{R3}
\end{align}
By a slight abuse of notation, we will use in the following sections the same letter  to denote 
a generator of the  coordinate ring $\CCq$ and its representation as a Hilbert space operator.

We adopt the convention that 
$\N= \{ 1,2, \ldots \}$ and $\N_0= \{0,1,2, \ldots\}$.  
Given an at most countable index set $I$ and a Hilbert space  $\hH_0$, consider the orthogonal sum 
$\hH=\mathop{\oplus}_{i\in I}\hH_0$. We write $\eta_i$ for the vector in $\hH$ which has the 
element $\eta\in\hH_0$ as its $i$-th component and zero otherwise. \label{Hn}
It is understood that $\eta_i=0$ whenever $i\notin I$. 

For a subset $A\subset [0,\infty)$,  the indicator function $\chi_{A}: [0,\infty)\ra \C$  
is defined by 
\[\label{chi}
\chi_A\hsp(t) := \left\{  
\begin{array}{l l}
    1,  & \quad t\in A\,,\\
    0, & \quad t\notin A\,.
\end{array}\right.
\]

For a subset $S$ of a *-algebra, the symbol  $\alg(S)$ 
stands for the *-subalgebra generated by the elements of $S$.

\section{Hilbert space representations of the 2-dimensional quantum complex plane}

In this section, we give a complete description of ``good'' *-representations of $\CCq$. 
Here ``good'' means that, 
in order to avoid pathological cases, we impose in Definition \ref{wb} 
some natural regularity conditions on the unbounded operators. 
These representations will be called \emph{well-behaved}, see \cite{S}. 
To motivate the regularity conditions, we start with formal algebraic manipulations.  
These algebraic relations, together with the regularity conditions of Definition~\ref{wb}, 
will allow us to classify  in  Theorem \ref{reps} 
all well-behaved representations of $\CCq$. 

Let $z_1$ and $z_2$ be  densely defined closed operators on a Hilbert space $\hH$ satisfying 
the relations \eqref{R1}--\eqref{R3} on a common dense domain. 
Set  $Q:= z_2^*z_2$. From the relations \eqref{R1}--\eqref{R3} within the algebra $\CCq$, we get 
\[  \label{zQ}
z_1 Q= Q z_1, \quad z_1^* Q= Q z_1^*, \quad z_2Q=q^2 Q z_2, \quad z_2^*Q=q^{-2} Q z_2^*\,, 
\]
and for any polynomial $p$ in one variable, \eqref{zQ} yields 
\[ \label{pQ}
z_1 p(Q)= p(Q) z_1, \quad z_1^* p(Q)= p(Q) z_1^*, \quad z_2 p(Q)=p(q^2 Q) z_2, \quad z_2^*p(Q)=p(q^{-2} Q) z_2^*. 
\]
Let us assume that \eqref{pQ} holds for all 
bounded Borel measurable functions on $\spec(Q)$, 
where $p(Q)=\int p(\lambda) \,\dd E(\lambda)$ is defined by the spectral theorem with the unique 
projection-valued measure $E$ of $Q$.     
Then $\ker(Q)=E(\{0\})\hH$ and $\ker(Q)^\perp=E((0,\infty))\hH$ are invariant under the actions of the generators 
of $\CCq$. 
On $\ker(Q)$, we have $Q= z_2^*z_2=0$, thus $z_2=z_2^*=0$, and \eqref{R3} becomes 
\[ \label{R30}
z_1 z_1^* = q^2 z_1^* z_1.  \tag{\ref{R3}'}
\]
On $\ker(Q)^\perp$, the operator $\sqrt{Q}^{\hs -1}= \int \frac{1}{\sqrt{\lambda}} \,\dd E(\lambda)$ 
is well-defined. Consider at the moment the abstract element 
\[
 w:= \sqrt{Q}^{-1} z_1= z_1 \sqrt{Q}^{-1} .   \label{w}
\]
Inserting \eqref{w} into 
the second relation of \eqref{R3} yields formally 
\[ \label{wwQ}
ww^* -q^2 w^*w = -(1-q^2)  z_2^* z_2 Q^{-1} = -(1-q^2). 
\]

Note that, by \eqref{zQ}, we have $z_1^*z_1\, z_2^*z_2 = z_2^*z_2\, z_1^*z_1$. This relation together with 
\eqref{wwQ}, \eqref{pQ}, \eqref{R30}, and 
the first equation  in \eqref{R3} motivate the following definition of 
well-behaved *-representations of $\CCq$. 

\begin{defn} \label{wb}
A well-behaved *-representations of $\CCq$ is given by densely defined  closed operators $z_1$ and $z_2$ 
satisfying \eqref{R1}--\eqref{R3} on a common dense domain such that 
\begin{enumerate}[(i)]
\item \label{0}
The self-adjoint operators $z_1^*z_1$ and $z_2^*z_2$ strongly commute. 
\item  \label{i}
$z_2$ is a $q$-normal operator, i.e., it satisfies the operator equation 
$$
z_2 z_2^* = q^2 z_2^* z_2.
$$
\item  \label{ii}
For all 
bounded Borel measurable functions $f$  on $\spec(Q)$, the operator relations 
$$
f(Q) z_1\subset z_1 f(Q),  \quad f(Q) z_1^*\subset z_1^* f(Q), \quad 
f(Q) z_2\subset z_2 f(q^2Q),  \quad f(Q) z_2^*\subset z_2^* f(q^2Q) 
$$
hold. 
\item   \label{iii}
On $\ker(Q)$, $z_1$ is a $q$-normal operator, i.e., 
$$
z_1 z_1^* = q^2 z_1^* z_1.
$$
\item On $\ker(Q)^\perp$, 
$z_1$ commutes with $\sqrt{Q}^{-1}$ and setting 
$w:= \sqrt{Q}^{-1} z_1= z_1 \sqrt{Q}^{-1} $ defines a densely defined closed operator 
fulfilling the operator equation  
\[ \label{ww}
ww^*  = q^2 w^*w -(1-q^2)  . 
\]
\end{enumerate}
Here, the equality of operators on both sides of the equations includes the equality of their domains. 
\end{defn}

The well-behaved representations of 
$q$-normal operators have been studied in \cite{CSS} and \cite{CW}. 
By \cite[Corollary 2.2]{CW}, 
any $q$-normal operator $\zeta$ on a Hilbert space $\G$ 
admits the following representation: 
\[ \label{zeta}
\G= \Ker(\zeta) \oplus ( \oplus_{n\in\Z} \,\G_0), \qquad  
\zeta =0 \ \ \text{on} \ \ \Ker(\zeta), \qquad 
\zeta\, g_n = q^n Z g_{n-1}\ \ \text{on} \ \ \oplus_{n\in\Z} \G_0, 
\]
where $Z$ denotes a self-adjoint operator on $\G_0$ sucht that 
$\spec(Z)\subset [q,1]$ and $q$ is not an eigenvalue of $Z$.

Furthermore, the representations of operators $w$ satisfying \eqref{ww} 
on a Hilbert space $\G$  have been classified in \cite[Lemma 2.3]{KW}. 
It follows from this lemma  that  $\G$  can be written as a direct sum 
$\G= \oplus_{m\in\N}\, \G_0$,  and the actions of $w$ and $w^*$ are determined by 
\[ \label{wg}
w\, g_m = \sqrt{q^{-2m}-1}\, g_{m+1} ,  \quad  w^* g_m = \sqrt{q^{-2(m-1)}-1}\, g_{m-1} ,       \quad g\in\G_0,\ \  m\in\N. 
\] 

Equations \eqref{zeta} and \eqref{wg} are all we need for the classification of the well-behaved repre\-sen\-ta\-tions of $\CCq$.

\begin{thm} \label{reps}
Any well-behaved Hilbert space representation of $\CCq$ is unitarily equivalent to a representation 
given by the following formulas: 
Let $\hH_{0}$, $\hH_{00}$ and $\cN$ be Hilbert spaces, and let 
$A$ and $B$ be self-adjoint operators on $\hH_{0}$ and $\hH_{00}$, respectively, 
such that their spectrum belongs to $[q,1]$ and $q$ is not an eigenvalue. 
Then the Hilbert space $\hH$ of the representation decomposes into the direct sum 
$$
\hH= \cN\oplus (\oplus_{k\in\Z} \hH_0)\oplus (\oplus_{n\in\Z} \oplus_{m\in \N}  \hH_{00}), 
$$
and the actions of $z_1$ and $z_2$ are determined by 
\begin{align}
&z_1=z_2=0 \ \ \text{on} \ \ \cN,  \label{N}  \\
&z_1\,h_k = q^k A \hs h_{k-1}, \ \ z_2 =0 \ \ \text{on} \ \  \oplus_{k\in\Z} \hH_0, \label{K}\\
&z_1\,h_{n,m} = \sqrt{q^{-2m}-1}\hs q^n B\hs  h_{n,m+1}, \ \ z_2\, h_{n,m} = q^n B\hs  h_{n-1,m}   \label{H} 
\ \ \text{on} \ \ \oplus_{n\in\Z} \oplus_{m\in \N} \hH_{00}. 
\end{align}
A common dense domain is obtained by considering the subspace of those elements of $\hH$ which 
have at most a finite number 
of non-zero components in the direct sum.  
Only the representation \eqref{H} is faithful. 
A representation is irreducible if and only if one of the Hilbert spaces $\hH_{0}$, $\hH_{00}$ and $\cN$  
is isomorphic to $\C$ and the others are zero. 
\end{thm}
\begin{proof}
As the sets $\{0\}$ and $(0,\infty)$ are invariant under multiplication with powers of $q$, 
it follows from Definition \ref{wb}\eqref{ii} that $\K := E(\{0\})\hH$ and 
$\G := E((0,\infty))\hH$ are invariant under the actions of $z_1$ and $z_2$. 
Clearly, $\hH= \K\oplus \G$. 
Since $\K=\ker(Q) =\ker(z_2^*z_2)$, we have $z_2=0$ on $\K$. 
By Definition \ref{wb}\eqref{iii}, the restriction of $z_1$ to $\K$  is a $q$-normal operator, 
therefore its representation is given by \eqref{zeta}. 
Setting $\cN:=\ker(z_1)$,  $\hH_0:= \G_0$ and $A:=Z$, we obtain \eqref{N} y \eqref{K} from \eqref{zeta}. 

By Definition \ref{wb}\eqref{i} and the definition of $\G$, $z_2$ is a $q$-normal operator on 
$\G$ with $\Ker(z_2) = \{ 0\}$. Therefore $z_2$ acts on $\G=\oplus_{n\in\Z} \hs \G_0$ 
by the formulas  on right hand side of \eqref{zeta}. 
Note that 
\[  \label{Qg}
Q \hs g_n =q^{2n} Z^2 g_n \quad  \text{on} \quad  \G_n:=\{ g_n : g\in \G_0\}, 
\]
with $\spec(q^{2n} Z^2)\subset [q^{2n+2}, q^{2n}]$ and $q^{2n+2}$ 
is not an eigenvalue.  
Considering the disjoint union $(0,\infty)=\cup_{n\in\Z} \,(q^{2n+2}, q^{2n}]$, 
one readily sees that $\G_n=E((q^{2n+2}, q^{2n}])\hs \G$.  
From Definition \ref{wb}\eqref{ii}, it follows that 
$$
E((q^{2n+2}, q^{2n}]) z_1\subset z_1 E((q^{2n+2}, q^{2n}]), \ \; 
E((0,\infty)\hsp \setminus\hsp  (q^{2n+2}, q^{2n}]) z_1\subset 
z_1E((0,\infty)\hsp\setminus \hsp (q^{2n+2}, q^{2n}]),  
$$
and the same holds for $z_1$ replaced by $z_1^*$. 
Since  $ \sqrt{Q}^{-1} $ trivially commutes with $E((q^{2(n+1)}, q^{2n}])$, 
we conclude that $w:=  \sqrt{Q}^{-1} z_1$  and $w^*$ leave $\G_n$ invariant. 
On $\G_n$, $w$ still satisfies \eqref{ww}, thus its representation is given by \eqref{wg}. 
Therefore we can write $\G_n = \oplus_{m\in\N}\hs  \hH_{n0}$ and 
\[ \label{whn}
w\hs h_{n,m} =  \sqrt{q^{-2m}-1}\hs h_{n,m+1}, 
\]  
where $h_{n,m}$ belongs to the $m$-th position in the direct sum $\oplus_{m\in\N} \hH_{n0}$. 
But $\G_n$ is just a copy of $\G_0$, 
so $\hH_{n0} = \hH_{00}$ for all $n\in\Z$. Equation \eqref{whn} yields 
\[
w^*w\hs h_{n,m}  = (q^{-2m}-1)\hs  h_{n,m} , 
\]
hence  $\hH_{nm}  := \{ h_{n,m} : h\in \hH_{00}\}$ 
is the eigenspace  for the eigenvalue $q^{-2m}-1$ 
of the restriction of $w^*w$ to $\G_n$.  
Definition \ref{wb}\eqref{0} implies that $w^*w$ and $Q$ strongly commute. 
Therefore the restrictions of  $w^*w$ and $Q$ to $\G_n=E((q^{2n+2}, q^{2n}])\G$ also 
strongly commute. As a consequence, the self-ad\-joint operator $Z$ from \eqref{Qg} leaves 
the eigenspaces $\hH_{nm}$ invariant. Denote the restriction of $Z$ to $\hH_{00}$ by $B$. 
Since $\hH_{nm}$ is an identical copy of $\hH_{00}$ in the $m$-th position of the direct sum $\oplus_{m\in\N} \hH_{n0}$, we get 
\[ \label{ZA} 
 Z \hs h_{n,m} =  B\hs  h_{n,m}  \quad \text{for all} \ \ h_{n,m}  \in\hH_{nm}. 
\]
Moreover, $B$ inherits the spectral properties  from $Z$ as required in Theorem \ref{reps}. 
Finally, \eqref{zeta} and \eqref{ZA} give 
\[   \label{z2}
z_2 \hs h_{n,m}  = q^n Z \hs h_{n-1,m}  = q^n B \hs h_{n-1,m},   \quad \text{} \ \ h_{n,m}  \in\hH_{nm}, 
\]
and from \eqref{Qg} and \eqref{whn}, we get 
$$
z_1 \hs h_{n,m} =  \sqrt{Q}\hs w \hs h_{n,m}  
=\sqrt{q^{-2m}-1} \sqrt{q^{2n} Z^2}\hs  h_{n,m+1}  = \sqrt{q^{-2m}-1}\hs q^n B h_{n,m+1} 
$$
for all $h_{n,m}  \in\hH_{nm}$. This proves \eqref{H}. 

That the representation \eqref{H} is faithful follows from the fact that the representations of 
$z_2$ and  $\omega$ are faithful, see \cite{CSS} and  \cite{KW}, respectively. 
The statement about irreducible representations is obvious since writing any of 
the Hilbert spaces $\hH_{0}$, $\hH_{00}$ and $\cN$ as an orthogonal sum of two non-zero 
subspaces will result in an orthogonal sum of non-trivial representations. 
\end{proof}

\section{Hilbert space representations on function spaces}

Note that the decomposition of $\hH$ in Theorem \ref{reps} is determined by the spectral properties of the self-adjoint 
operators $Q$ and $w^*w$. 
As well-known \cite[Theorem VII.3]{RS},  each self-adjoint operator $T$ on a separable Hilbert space is 
unitarily equivalent to a direct sum of multiplication operators on $\LL(\spec(T), \mu)$.   
We will use this fact to realize the representations of  Theorem \ref{reps} 
on $\LL$-spaces which will be the basis for studying function algebras on 
the 2-dimensional quantum complex plane in the next section. 

The direct sum of Hilbert spaces $\oplus_{n\in\Z} \oplus_{m\in \N}  \hH_{00}$ in Theorem \ref{reps}  
is isomorphic to the tensor product $ \lN \otimes (\oplus_{n\in\Z}  \hH_{00}) $. 
Let $\zeta$ be a $q$-normal operator acting on $\oplus_{n\in\Z}  \hH_{00}$ by the formulas on the right hand side of \eqref{zeta},  
and let $\omega$ act on $\lN$ by the formulas in \eqref{wg} with $\G_0=\C$. 
Then $z_2$ from \eqref{z2} and $w$ from \eqref{whn} can be written $z_2= \id \otimes \zeta$ and $w=\omega\otimes\id$, 
respectively. 
It has been shown in  \cite[Theorem 1]{CSS} that any $q$-normal operator $\zeta$ is unitarily 
equivalent to a direct sum of operators of the following form: 
There exists a $q$-invariant Borel measure
$\mu$ on $[0,\infty)$
such that  $\zeta$ and $\zeta^*$ act 
on $\hH= \LL([0,\infty), \mu)$ by 
\[ \label{L2rep}
\zeta \hs f(t)=q\hs t f(qt),\ \ \zeta^*f(t)=tf(q^{-1}t),\ \   f\hsp \in\hsp\dom(\zeta)\hsp:=\hsp 
\{ h\hsp \in\hsp \LL([0,\infty), \mu): \mbox{$\int$}\hs t^2\hs |h(t)|^2 \dd\mu(t)\hsp<\hsp \infty\}. 
\]
Here, the $q$-invariance of the measure means that 
$\mu(qS) = \mu(S)$ for all Borel subsets $S$ of $[0,\infty)$. 
Note that $\ker(\zeta)=\{0\}$  if and only if $\mu(\{0\}) =0$. 
Therefore, in order to obtain a representation of the form \eqref{H}, 
we have to assume that $\mu(\{0\}) =0$. 

To turn $\lN$ into an $\LL$-space, we consider the operator 
$y:=\sqrt{\omega^*\omega +1}$ on $\lN$. Denoting by $\{ e_n : n\in\N\}$ the standard basis of $\lN$, we have 
\[ \label{y}
 y\hs e_n = q^{-n}\hs e_n,\quad n\in\N. 
\]
Since the set of eigenvalues of $y$  is discrete, $y$ can be realized as a multiplication operator 
on $\LL(\spec(y),\sigma) \cong \lN$ 
by choosing the counting measure $\sigma(\{q^{-n}\}) = 1$ 
on $\spec(y)$. Extending $\sigma$ to a Borel measure on $[0,\infty)$ by setting $\sigma([0,\infty)\setminus \spec(y)) :=0$, 
we get 
$y\hs g(s)  = s\hs g(s)$. 
The set 
$$
\{ e_n:=\chi_{\{q^{-n}\}}(s): n\in\N\}
$$
is an orthonormal basis of $\LL(\spec(y),\sigma)$, 
where $\chi_{\{q^{-n}\}}$ denotes the indicator function \eqref{chi}.  
Note that 
$$
\chi_{\{q^{-n}\}}\hsp(qs)= \chi_{\{q^{-(n+1)}\}}\hsp(s)=e_{n+1}\ \ \text{and}  \ \ 
\sqrt{(q\hs s)^{2} -1}\hs \chi_{\{q^{-(n+1)}\}}\hsp(s) = \sqrt{ q^{-2n} -1}\hs \chi_{\{q^{-(n+1)}\}}\hsp(s), 
$$ 
where we used $f(t)\hs  \chi_{\{t_0\}}\hsp(t) =f(t_0)\hs  \chi_{\{t_0\}}\hsp(t)$ in the second equation. Hence 
\[   \label{wy} 
\omega \hs g(s) = \sqrt{(q\hs s)^{2} -1}\, g(q\hs s) \quad \text{and} \quad 
\omega^* g(s) = \sqrt{s^{2} -1}\, g(q^{-1}\hs s)
\]
for $g\in\dom(\omega)=\dom(\omega^*)
:=\{ h\in \LL([0,\infty),\sigma) \hs:\hs  \mbox{$\int$}\hs s^{2}|h(s)|^2\hs  \dd\sigma(s) < \infty\})$. 
En particular, 
\[
\omega^*e_1= \sqrt{s^{2} -1}\, \chi_{\{q^{-1}\}}\hsp(q^{-1}\hs s) = \sqrt{1^2 -1}\, \chi_{\{q^{-1}\}}\hsp(q^{-1}\hs s) =0, 
\]
as required. Also, although $||\chi_{\{1\}}\hsp(s)||=0$ and $\chi_{\{1\}}\hsp(qs)= \chi_{\{q^{-1}\}}\hsp(s) =e_1$, 
we have 
$$
\sqrt{(q\hs s)^{2} -1}\, \chi_{\{1\}}\hsp(qs) =\sqrt{(q\hs q^{-1})^{2} -1}\, \chi_{\{1\}}\hsp(qs) =0, 
$$
so that \eqref{wy} remains consistent.

Now, 
under the isomorphism 
$\LL( [0,\infty),\sigma)\ot \LL([0,\infty),\mu) \cong 
\LL( [0,\infty)\! \times\! [0,\infty), \sigma\otimes\mu )$, 
we obtain from \eqref{L2rep} and \eqref{wy} the following representation of 
 $z_1 = \sqrt{Q}\hs w =   \omega\ot \sqrt{\zeta^*\zeta}$  and $z_2 = \id\otimes \zeta$, 
\[ \label{zL2}
z_1\hs h(s,t) =  \sqrt{(qs)^{2}-1}\hs t\hs h(q s,t),    \qquad z_2\hs g(s,t)=q \hs t\hs g(s, q\hs t), 
\] 
where $h\in \dom(\omega) \ot_{\mathrm{alg}} \dom(\zeta) $ and 
$g\in \LL( [0,\infty),\sigma)\ot_{\mathrm{alg}}\dom(\zeta)$. 
To sum up, we have shown that the representations from \eqref{H} 
are unitarily equivalent to a direct sum of representations 
of the type described in  \eqref{zL2}. 

Recall that $\cN\oplus (\oplus_{k\in\Z} \K_0)$ in Theorem \ref{reps} 
corresponds to the kernel of the $q$-normal operator $z_2$, 
and that a  $q$-normal operator in the representation \eqref{L2rep} 
has a trivial kernel if and only if $\mu(\{0\})=0$.  
Since $\mu$ from the last paragraph was assumed to satisfy $\mu(\{0\})=0$, 
we will now add a point measure $\delta_0$ centred at $0$ to it. 
By unitary equivalence, we may assume that $\delta_0(\{0\})=1$. 
Then the representation of $z_1$ on $\oplus_{k\in\Z} \K_0$ is again 
unitarily equivalent to a direct sum of  representations of the type \eqref{L2rep}. 
To realize these representations on our $\LL$-space, we choose a 
$q$-invariant measure on $[0,\infty)$, say $\nu$, assume again $\nu(\{0\})=0$, 
take the product measure $\nu \ot \delta_0$, 
and add it to $\sigma\otimes \mu$. 
Then 
$$
\LL( [0,\infty)\! \times\! [0,\infty), \sigma\hsp \ot\hsp\mu \hsp+\hsp \nu  \hsp \ot\hsp \delta_0 ) 
\cong \LL( [0,\infty)\! \times\! [0,\infty), \sigma \hsp \ot\hsp\mu )\oplus \LL( [0,\infty)\! \times\! [0,\infty), \nu\hsp \ot\hsp\delta_0 ), 
$$ 
and on $\LL( [0,\infty)\! \times\! [0,\infty),\nu\hsp \ot\hsp \delta_0 )$, we have the representation 
\[ \label{zL20}
z_1\hs h(s,t) =q\hs  s\hs  h(q s, t) =q\hs  \chi_{\{0\}}\hsp(t)\hs s\hs  h(q\hs  s,t),    \quad z_2\hs g(s,t)=q\hs  t\hs g(s,q\hs t) =0 
\] 
for all  $h$    
such that $\int s^2 |h(0,s)|^2 \hs \dd\nu(s)<\infty$ 
and for all $g$. 
Here, for functions depending on the second variable $t$,  we used the fact that $\mathrm{supp}(\delta_0)=\{0\}$. 
Again, by  \cite[Theorem 1]{CSS} and  the same argumentation as above,  
the representations from \eqref{K} are unitarily equivalent to 
a direct sum of representations of the type described in \eqref{zL20}.

Finally, to obtain a non-trivial component  $\cN= \ker{z_1}\cap \ker{z_2}$, we add to 
$\sigma\hsp \ot\hsp\mu  \hsp+\hsp \nu\hsp \ot\hsp\delta_0 $ 
the point measure $\epsilon \hs \delta_0\ot \delta_0$, where $\epsilon=0$ or $\epsilon=1$ 
depending on whether $\cN=\{0\}$ or $\cN\neq \{0\}$. Summarizing, we have proven the following theorem: 

\begin{thm} \label{teo2} 
The Hilbert space representations of $\CCq$ from Theorem \ref{reps} are 
unitarily equivalent to a direct sum of representations of the following type: 
Let $\mu$ and $\nu$ be $q$-invariant Borel measures on $[0,\infty)$ 
such that $\mu(\{0\})=\nu(\{0\})=0$. 
Denote by 
$\delta_0$ the Dirac measure centred at $0$, and 
define a 
Borel measure $\sigma$ on 
$[0,\infty)$ by setting $\sigma(\{q^{-n}\}): = 1$ for all $n\in \N$ and 
$\sigma\big([0,\infty) \setminus \{q^{-n}:n\in\N\}\big):=0$. 
For $\epsilon \in\{0,1\}$, 
consider the Hilbert space 
\[ \label{HL2} 
\hH:= \LL\big( [0,\infty)\! \times\! [0,\infty)\hs , \hs  \sigma\hsp \ot\hsp\mu\hsp+\hsp\nu\hsp \ot\hsp \delta_0 + \epsilon\hs \delta_0\ot \delta_0 \big),  
\]
and set 
$$
\dom(z_1):= \{h \in \hH \,:\,  s\hs t\hs  h \in\hH \ \text{and} \ s\hs  \chi_{\{0\}}(t)\hs  h \in\hH\} , \quad 
\dom(z_2):= \{g \in \hH \,:\, t\hs g\in\hH\} \,, 
$$
where $\chi_{\{0\}}$ denotes the indicator function from \eqref{chi}. 
For $h\in \dom(z_1)$ and $g\in\dom(z_2)$, the ac\-tions of the generators of $\CCq$
are given by 
\begin{align} \label{multrep}
z_1\hs h(s,t) &= \sqrt{(qs)^{2}-1}\hs  t\hs h(q s,t) + q\hs \chi_{\{0\}}\hsp(t)\hs s\hs  h(q s,t),   &
z_2\hs g(s,t)&= q\hs t\hs g(s,q\hs t) , \\
z_1^*\hs h(s,t) &= \sqrt{s^{2}-1}\hs  t\hs h(q^{-1} s,t) +  \chi_{\{0\}}\hsp(t)\hs s\hs  h(q^{-1} s,t),   &
z_2^*\hs g(s,t)&= t\hs g(s,q^{-1} t).  \label{multrep*}
\end{align}
\end{thm}
Note that 
\begin{align} \nonumber
\hH&= 
\LL\big( [0,\infty)\! \times\! [0,\infty)\hs , \hs \epsilon\hs \delta_0\ot \delta_0 \big)\oplus 
\LL\big( [0,\infty)\! \times\! [0,\infty)\hs , \hs\nu  \hsp \ot\hsp \delta_0\big) \oplus 
\LL\big( [0,\infty)\! \times\! [0,\infty)\hs , \hs \sigma \hsp \ot\hsp\mu\big)\\
&=
\LL\big( \{0\}\!\times\! \{0\}\hs , \hs \epsilon\hs \delta_0\ot \delta_0 \big)\oplus 
\LL\big( [0,\infty)\! \times\!\{0\} \hs , \hs  \nu \hsp \ot\hsp\delta_0 \big) \oplus 
\LL\big( \{q^{-n}:n\hsp\in\hsp\N\} \! \times\! [0,\infty)\hs , \hs \sigma \hsp \ot\hsp\mu\big)  \label{ortoH}
\end{align} 
and that the restriction of the representation \eqref{multrep} to one of the orthogonal components  
cor\-res\-ponds to one of the representations from  \eqref{N}--\eqref{H}. Of course, we could have formulated 
Theorem \ref{teo2}  for each of the orthogonal subspaces separately. 
The reason why we prefer to work with a single Hilbert space on the domain 
$[0,\infty)\! \times\! [0,\infty)$ 
 will become clear in the next section. 

\section{C*-algebra of continuous functions vanishing at infinity}  
The aim of this section is to define a C*-algebra which can be viewed as the algebra of continuous functions 
on the 2-dimensional quantum complex plane
vanishing at infinity. The definition will be motivated by a similar construction for the 
1-dimensional quantum complex plane \cite{CW}. 
As a point of departure, we first  look for an auxiliary *-algebra, where the commutation relations are 
considerable simple. 

For the convenience of the reader, we recall the construction of the C*-algebra $C_0(\C_q)$ 
of continuous functions vanishing at infinity on the 1-dimensional quantum complex plane \cite{CW}. 
Given a representation of the type \eqref{L2rep}, consider the following *-subalgebra of $\rmB( \LL([0,\infty), \mu))$\hs:
\[  \label{Cq}
\alg\{C_0(\spec(|\zeta|),U\}   :=  \Big\{  \sum_{\text{finite}} f_k(|\zeta|) \hs U^k\,:\, 
k\hsp \in\hsp \Z,\  f_k\hsp\in\hsp C_0( \spec(|\zeta|),\ f_k(0)\hsp=\hsp 0 \ \text{if} \ k\hsp\neq\hsp 0\Big\}, 
\]
where $\mu(\{0\})=0$ and $U$ denotes the unitary operator from the polar decomposition $\zeta =U\hs |\zeta|$. 
For all bounded continuous functions $f$ on $\spec(|\zeta|)$, the operators $f(|\zeta|)$ and $U$ satisfy the commutation relation 
$$
U\hs f(|\zeta|) = f(q\hs |\zeta|) \hs U. 
$$
In\cite{CW}, 
a representation of the type   \eqref{L2rep}  of a $q$-normal operator $Z=U\hs |Z|$ was said to be  
\emph{universal} if $\spec(|Z|)=[0,\infty)$, or equivalently if $\supp(\mu)=[0,\infty)$. 
It has the universal property that 
\[
\alg\{C_0(\spec(|Z|),U\}   \ni \sum_{\text{finite}} f_k(|Z|) \hs U^k \ 
\longmapsto \ \sum_{\text{finite}} f_k(|\zeta|) \hs U^k \in \alg\{C_0(\spec(|\zeta|),U\}   
\] 
yields always a well-defined surjective *-homomorphism. 
Although the exact definition in \cite{CW} is slightly abstract, \cite[Theorem 3.3]{CW} states that 
$C_0(\C_q)$ is isomor\-phic to the norm closure of $\alg\{C_0(\spec(|Z|),U\}$ 
in $\rmB( \LL([0,\infty), \mu))$.

Motivated by the previous description, we call a representation from Theorem \ref{teo2} 
\emph{universal} if $\epsilon=1$ and 
$\supp(\mu)=\supp(\nu)=[0,\infty)$. Such $q$-invariant measures can be obtained, for instance, 
by taking the Lebesgue measure $\lambda$ on $(q,1]$ and setting 
$$ 
\mu(M) = \sum_{k\in\Z} \lambda ( q^{-k}(M\cap (q^{k+1},q^k])) .
$$

Given a universal representation, consider the polar decompositions 
$z_1\hsp=\hsp U\hs |z_1|$ and $z_2\hsp  =\hsp V\hs |z_2|$. 
For all  $h\in \dom(|z_1|)=\dom(z_1)$ and $g\in \dom(|z_2|)=\dom(z_2)$, \eqref{multrep} 
and \eqref{multrep*} imply 
\[ \label{mod}
|z_1|\hs h(t,s) = \big(\chi_{[q^{-1},\infty)}\hsp(s) \sqrt{s^{2}-1}\hs  t +  s \hs \chi_{\{0\}}\hsp(t)\big)   h(t, s),   \qquad 
|z_2|\hs g(t,s)= \hs t\hs g( t,s).         
\] 
Since 
\begin{align*}
&\ran(|z_1|)=\ker(|z_1|)^\perp= \ran\big(\chi_{(0,\infty)}\hsp(t)\,\chi_{[q^{-1},\infty)}\hsp(s) + \chi_{\{0\}}\hsp(t)\, \chi_{(0,\infty)}(s)\big), \\
&\ran(|z_2|)=\ker(|z_2|)^\perp= \ran(\chi_{(0,\infty)}\hsp(t)), 
\end{align*}
it follows from \eqref{multrep} that 
\begin{align} \label{U}
U h(s, t) &=\big( \chi_{(0,\infty)}\hsp(t)\,\chi_{[q^{-1},\infty)}\hsp(qs) + \chi_{\{0\}}\hsp(t)\, \chi_{(0,\infty)}(s)\big)h(qs, t) , \\
V h(s, t) &= \chi_{(0,\infty)}\hsp(t)  h(s, qt),    \label{V}
\end{align}
for all $h\in \hH$, 
where we used $\chi_{(0,\infty)}\hsp(q\hs r)=\chi_{(0,\infty)}\hsp(r)$. 
Their adjoints act on $\hH$ by 
\begin{align}  \label{U*}
U^* h(s, t) &=\big( \chi_{(0,\infty)}\hsp(t)\,\chi_{[q^{-1},\infty)}\hsp(s) + \chi_{\{0\}}\hsp(t)\, \chi_{(0,\infty)}(s)\big)h(q^{-1} s, t) , \\
V^* h(s, t) &= \chi_{(0,\infty)}\hsp(t)  h(s, q^{-1}t) .  \label{V*}
\end{align}
From \eqref{U}--\eqref{V*}, we get 
\begin{align}  \label{UU*}
U^* U  &=\chi_{(0,\infty)}\hsp(t)\,\chi_{[q^{-1},\infty)}\hsp(s) + \chi_{\{0\}}\hsp(t)\, \chi_{(0,\infty)}(s),\\
 U U ^*  &=\chi_{(0,\infty)}\hsp(t)\,\chi_{[q^{-1},\infty)}\hsp(qs) + \chi_{\{0\}}\hsp(t)\, \chi_{(0,\infty)}(s),  \label{U*U}  \\
V^* V &= \chi_{(0,\infty)}\hsp(t)  =V V^*.   \label{VV*}
\end{align}
Using again $\chi_{(0,\infty)}\hsp(q^{\pm 1}t)=\chi_{(0,\infty)}\hsp(t) $  and $\chi_{\{0\}}\hsp(q^{\pm 1} t)=\chi_{\{0\}}\hsp(t)$, 
one easily sees that 
\[
UV=VU, \quad UV^*=V^* U, \quad U^*V=VU^*,\quad U^*V^*=V^*U^*. 
\]
Considering Borel measurable functions $f$ on $[0,\infty)\! \times\! [0,\infty)$ as multiplication operators by 
$$
f \hs h(s,t):= f(s,t)\, h(s,t), 
$$
we obtain from \eqref{U}--\eqref{V*} the following simple commutation relations: 
\begin{align}   \label{fUU}
U\hs f(s,t)&= f(qs, t)\hs U, &  U^*\hs f(s,t) &= f(q^{-1}s, t)\hs U^*, \\
V\hs f(s,t)&= f(s, qt)\hs V,  & V^*\hs f(s,t)  &=f(s, q^{-1} t) \hs  V . \label{fVV}
\end{align}
In fact, the reason for choosing $|\zeta|$ from \eqref{zeta} and $y$ from \eqref{y} as multiplication operators was 
to obtain such simple commutation relations between functions and the phases from the polar decom\-positions of 
the generators of $\CCq$.  As a consequence, 
\[ \label{fun}
\mathrm{Fun}(\C_q^2) := \Big\{ \sum_{\text{finite}} f_{nm}(s,t)\hs U^{\#n}V^{\#m}\,:\, f\in \mathcal{L}_{\infty}([0,\infty)\times [0,\infty))\Big\} 
\]
is a *-subalgebra of $\rmB(\hH)$, where 
\[\label{hash}
U^{\# n} := \left\{  
\begin{array}{l l}
    U^n,  & \quad n\geq 0\,,\\
    U^{*n}, & \quad n< 0\,, 
\end{array}\right.
\qquad 
V^{\# n} := \left\{  
\begin{array}{l l}
    V^n,  & \quad n\geq 0\,,\\
    V^{*n}, & \quad n< 0\,, 
\end{array}\right.
\qquad 
n\in\Z. 
\]
Moreover, by \eqref{mod} and the previous commutation relations, there exists for all $k,l,m,n\in\N_{0}$ 
a Borel measurable function 
$p_{klmn}$ on $[0,\infty)\times [0,\infty)$ 
such that 
\[  \label{p}
z_1^k z_1^{*l} z_2^m z_2^{*n} =  p_{klmn}(s,t)\hs U^{\# k-l}V^{\# m-n} \,. 
\]

Equations \eqref{fun} and \eqref{p} are the motivation for the construction of the C*-algebra of 
continuous functions on $\C^2_q$ vanishing at infinity. 
Before treating the quantum case, let us briefly review the classical C*-algebra $C_0(\C^2)$. 
In analogy to the polar decomposition of the generators, write 
$z_1= \E^{\im \varphi} |z_1|$ and $z_2= \E^{\im \theta} |z_2|$. Let $n,m\in\Z$. 
Given a function $f_{nm}\in C_0([0,\infty)\! \times\! [0,\infty))$, the assignment 
$$
\C^2 \ni (\E^{\im \varphi} |z_1| ,\E^{\im \theta} |z_2| )\  
 \longmapsto\  f_{nm}(|z_1|,  |z_2|)\hs \E^{\im \varphi n}\hs \E^{\im \theta m} \in\C 
$$
defines a function in $C_0(\C^2)$ if and only if  
\begin{enumerate}[(a)]
\item
$f_{nm}(0,  |z_2|)\hs \E^{\im \varphi n}\hs \E^{\im \theta m}$ 
does not depend on $\varphi $ 
\ \,$\Longrightarrow$ \ $f_{nm}(0,  |z_2|)=0$ for $n\neq 0$, 
\item
$f_{nm}(|z_1|,0)\hs \E^{\im \varphi n}\hs \E^{\im \theta m}$ 
does not depend on $\theta $ 
\ \;$\Longrightarrow$ \ \,$f_{nm}(|z_1|, 0)=0$ for $m\neq 0$. 
\end{enumerate}
Moreover, the following  *-subalgebra of $C_0(\C^2)$,  
\begin{align}  \label{C0C2} 
\mathcal{C}_0(\C^2):= \Big\{  \sum_{\text{finite}} f_{nm}(|z_1|,|z_2|) \hs \E^{\im \varphi n}\hs \E^{\im \theta m}&\,:\, \ 
 f_{nm} \in C_0([0,\infty)\! \times\! [0,\infty)), \ \,n,m\in\Z, \\[-8pt]
 &f_{nm}(0,|z_2|)=0 \ \;\text{if} \ \;n\neq 0, \ \  f_{nm}(|z_1|,0)=0 \ \;\text{if} \ \;m\neq 0 \, \Big\}, \nonumber 
\end{align}
separates the points of $\C^2$. By the Stone--Weierstra{\ss} theorem, its norm closure yields $C_0(\C^2)$. 

To pass from the classical to the quantum case, we start with a universal representation from Theorem \ref{teo2}. 
For all bounded continuous functions $g$ on $[0,\infty)$, the operators $g(|z_1|), \, g(|z_2|) \in\rmB(\hH)$  are well-defined 
by the spectral theorem,  and  \eqref{mod} shows that 
$$
g(|z_1|)\, h(s,t)  = g\big(\hs \chi_{[q^{-1},\infty)}\hsp(s)\sqrt{s^{2}-1}\hs  t +  s \hs \chi_{\{0\}}\hsp(t)\hs \big)\,h(s,t),\qquad 
g(|z_2|)\hs h(s,t)  = g(t) \hs h(s,t)  . 
$$
By the universality of the representation, we have $\|g(|z_i|)\| = \|g\|_\infty$, $i=1,2$. 
In particular, the norm does not depend on the chosen measures of a universal representation. 
Similarly, for $f\in C_0([0,\infty)\! \times\! [0,\infty))$, the formula  
\[     \label{fzz}
f(|z_1|,|z_2|)\, h(s,t) := f\big(\hs \chi_{[q^{-1},\infty)}\hsp(s)\sqrt{s^{2}-1}\hs  t +  s \hs \chi_{\{0\}}\hsp(t) \hs,t\hs\big)\hs h(s,t),\qquad h\in\hH, 
\] 
yields a well-defined operator in $\rmB(\hH)$ with  $\|f(|z_1|,|z_2|)\|  =  \|f\|_\infty$. 
These observations lead to the following definition of $C_0(\C_q^2)$. 
\begin{defn}
Given a universal representation of $\CCq$ from Theorem \ref{teo2},  
let $z_1 = U\hs |z_1|$ and $z_2 = V\hs |z_2|$ be the polar decompositions of 
$z_1$ and $z_2$, respectively. 
The   C*-algebra $C_0(\C_q^2)$  of continuous functions 
on the 2-dimensional quantum complex plane
vanishing at infinity is defined as the norm closure of 
\begin{align} \label{defC0} 
\mathcal{C}_0(\C^2_q):=\text{*-}\mathrm{alg} \Big\{  &\sum_{\text{finite}}
 f_{nm}\big(\chi_{[q^{-1},\infty)}\hsp(s) \sqrt{s^{2}-1}\hs  t +  s \hs \chi_{\{0\}}\hsp(t)\, ,\hs t\hs \big) \, U^{\# n}\hs V^{\# m}\,:\, \ n,m\in\Z,    \\[-4pt]
& \   f_{nm} \in C_0([0,\infty)\! \times\! [0,\infty)),   \  \;  f_{nm}(0,t)=0 \ \,\text{if} \ \,n\neq 0, \ \;  f_{nm}(s,0)=0 \ \,\text{if} \ \,m\neq 0  \Big\} 
\nonumber
\end{align}
in $\rmB(\hH)$. 
\end{defn} 
Apart from the non-commutativity in \eqref{fUU} and \eqref{fVV}, the main difference to the classical case 
is the unusual expression in the first argument of the function $f_{nm}$. However, 
if we look at  the representation on the orthogonal components of \eqref{ortoH} separately, 
our formulas have a natural geometric interpretation. 
First note that the function $h_{00}(s,t):=\chi_{\{0\}}\hsp(s)  \chi_{\{0\}}\hsp(t) \in\hH$ generates the 1-dimensional invariant subspace 
$\LL( [0,\infty)\! \times\! [0,\infty)\hs , \hs \epsilon\hs \delta_0\ot \delta_0 )$, and the representation of $\mathcal{C}_0(\C^2_q)$  
on it reads 
 \[  \label{ev0}
  \sum_{\text{finite}} 
  f_{nm}(\chi_{[q^{-1},\infty)}\hsp(s) \sqrt{s^{2}-1}\hs  t +  s \hs \chi_{\{0\}}\hsp(t)\, ,\hs t) \, U^{\# n}\hs V^{\# m} \,   h_{00}(s,t) = f_{00}(0,0) \hs h_{00}(s,t) 
 \]
since $f_{nm}(0,0)= 0$ if $n\neq 0$ or $m\neq 0$.  
Obviously, \eqref{ev0} corresponds to evaluating functions on 2-dimensional complex plane at $(0,0)$. 
This 1-dimensional representation describes the only classical point $(0,0)$ of $\C_q^2$. 

Next, on $\LL( [0,\infty)\! \times\! [0,\infty)\hs , \hs\nu  \hsp \ot\hsp \delta_0)$, we have $z_2=0$ and can 
thus write 
\[ \label{C0C}
\sum_{n,m} f_{nm}(\chi_{[q^{-1},\infty)}\hsp(s) \sqrt{s^{2}-1}\hs  t +  s \hs \chi_{\{0\}}\hsp(t)\, ,\hs t) \, U^{\# n}\hs V^{\# m} 
= \sum_{n} f_{n0}(s ,0) \, U^{\# n}. 
\]
Recalling that $z_1$ acts on $\LL( [0,\infty)\! \times\! [0,\infty)\hs , \hs\nu  \hsp \ot\hsp \delta_0)$ as a $q$-normal operator  
and comparing \eqref{C0C} with \eqref{Cq}  shows that the restriction of 
$\mathcal{C}_0(\C^2_q)$ to $\LL( [0,\infty)\! \times\! [0,\infty)\hs , \hs\nu  \hsp \ot\hsp \delta_0)$ generates $C_0(\C_q)$. 
This representation corresponds to an inclusion $\C_q\times\{0\} \subset \C^2_q$. 

Finally, on $\LL\big( [0,\infty)\! \times\! [0,\infty)\hs , \hs \sigma \hsp \ot\hsp\mu\big)$, 
we have $t= |z_2|>0$ and $\chi_{[q^{-1},\infty)}\hsp(s) \sqrt{s^{2}-1} = |\omega |$, see \eqref{wy}. 
Thus the representation of the functions from \eqref{defC0} can be written 
$$
 \sum_{\text{finite}}f_{nm}(|\omega|\,  t \, ,\hs t) \, U^{\# n}\hs V^{\# m}. 
$$
Classically we get, for all $|\omega|\geq 0$, 
$$
\sum_{\text{finite}}f_{nm}(|\omega|\,  t \, ,\hs t) \hs \E^{\im \varphi n}\hs \E^{\im \theta m} \Big|_{t=0}
=  \sum_{\text{finite}}f_{nm}(0,0) \hs \E^{\im \varphi n}\hs \E^{\im \theta m} = f_{00}(0,0). 
$$
Therefore these functions separate only the points of $\C^2\setminus     \C \times\{0\} $ and the whole 
subspace  $\C \times\{0\} $ gets identified with the single point $(0,0)$. 
Geometrically, this corresponds to a 2-dimensional complex plane, where $\C \times\{0\} $ is shrunk 
to one point. 

Arguing backwards, we can say that the representation from \eqref{zL2} corresponds to a
2-dimen\-sional quantum complex plane, where  $\C_q \times\{0\} $ is shrunk to a point, and that 
$\C_q \times\{0\} $ gets glued into this space by the representation \eqref{zL20}. 
Moreover, the origin of the 2-dimensional quantum complex plane is the only classical point described 
by the 1-dimensional representation \eqref{ev0}.

\end{document}